\documentclass[11pt,reqno]{amsart}
\usepackage{graphicx}
\usepackage{cite}
\usepackage{amscd}
\usepackage{amsmath}
\usepackage{amsfonts}
\usepackage{setspace}
\usepackage{amssymb}
\newtheorem{theorem}{Theorem}[section]

\newtheorem{lemma}{Lemma}[section]
\newtheorem{definition}{Definition}[section]

\newcounter{romnum}

\newcounter{arabicnum}

\begin{document}
\title[On a discrete elliptic problem with a weight]
{On a discrete elliptic problem with a weight}

\author[M. Ousbika]{Mohamed Ousbika}
\address{Team of Modeling and Scientific Computing, Department of Mathematics and Computer, Multidisciplinary Faculty of Nador,
University Mohammed first, Morocco}
\email{nizarousbika@gmail.com}

\author[Z. El Allali]{Zakaria El Allali}	
\address{Team of Modeling and Scientific Computing, Department of Mathematics and Computer, Multidisciplinary Faculty of Nador,
University Mohammed first, Morocco}
\email{z.elallali@ump.ma}

\author[L. Kong]{Lingju Kong}
\address{Department of Mathematics,
University of Tennessee at Chattanooga, Chattanooga, TN 37403, USA.}
\email{Lingju-Kong@utc.edu}

\baselineskip 21pt

\begin{abstract}
Using the variational approach and the critical point theory, we established several criteria for the existence of at least one nontrivial solution for a discrete elliptic 
boundary value problem with a weight $p(\cdot, \cdot)$ and depending on a real parameter $\lambda$.
\end{abstract}


\maketitle

\setcounter{equation}{0}
\section{Introduction}
\label{sec1}

In this paper, we consider  the discrete elliptic boundary value problem with a weight
\begin{equation}\label{eqn1.1}
	\begin{cases}
	
	-\Delta_{1}(p(i-1,j)\Delta_{1}u(i-1,j))-\Delta_{2}(p(i,j-1)\Delta_{2}u(i,j-1))=\lambda f((i,j),u(i,j)),\\
	\hspace*{3.56in} \forall (i,j)\in [1,m]_{\mathbb{Z}}\times[1,n]_{\mathbb{Z}},\\
		u(0,j)=u(m+1,j)=0, \quad  \forall j\in [1,n]_{\mathbb{Z}},\\
		u(i,0)=u(i,n+1)=0, \quad  \forall i\in [1,m]_{\mathbb{Z}},
    \end{cases}
\end{equation}
where $[1,m]_{\mathbb{Z}}=\{1,\ldots , m\},  [1,n]_{\mathbb{Z}}=\{1,\ldots , n\}$, $\Delta_{1}u(i,j)=u(i+1,j)-u(i,j) $ 
and $ \Delta_{2}u(i,j)=u(i,j+1)-u(i,j)$ are the forward difference operators, 
$f:[1,m]_{\mathbb{Z}}\times[1,n]_{\mathbb{Z}}\times\mathbb{R}\rightarrow \mathbb{R}$ is a continuous function  
subject to some suitable assumptions, $\lambda$ is a positive parameter, and $ p:[0,m]_{\mathbb{Z}}\times[0,n]_{\mathbb{Z}}\rightarrow (0,+\infty)$ 
is a given  function such that 
 \begin{equation}\label{eqn1.2}
 p(0,j)=0, \quad \forall j \in [1,n]_{\mathbb{Z}},  \quad \text{and} \quad  p(i,0)=0, \quad \forall i \in [1,m]_{\mathbb{Z}}. 
\end{equation} 
 
The problem (\ref{eqn1.1}) can be regarded as the discrete counterpart of the elliptic partial differential equation
\begin{equation*} 
 \begin{cases}
 \dfrac{\partial}{\partial x}\left( g(x,y)\dfrac{\partial u}{\partial x}\right)+\dfrac{\partial}{\partial y}
 \left( g(x,y)\dfrac{\partial u}{\partial y}\right)  +\lambda f((x,y),u(x,y))=0,  \quad \forall(x,y)\in \Omega,\\[0.2cm]
 u(x,y)=0,\quad \forall(x,y)\in \partial\Omega.
 \end{cases}
\end{equation*}
 
As is well known, the study of nonlinear algebraic systems arise in a large variety of applications such as in
reaction-diffusion equations, neural networks, compartmental systems, and population models. Nonlinear algebraic systems 
can be obtained from several Dirichlet problems of differential and difference equations, three point boundary value problems, 
and steady states of complex dynamical networks. We refer the reader to \cite{Gyori} and the references therein for more information. 
      
Discrete elliptic problems involving functions with two or more discrete variables appear frequently in applications and are investigated in the literature. 
Recently, several works studied the existence and multiplicity of solutions for such problems. 
See, for example, \cite{Gao, Mihailescu, Shapour}. The progress of modern digital computing devices contributes greatly to the increasing interest
in discrete problems. In fact, because these problems can be simulated in a simple way by means of these devices and the simulations often reveal
important information about the behavior of complex systems,  many recent studies related to image processing, population models, 
neural networks, social behaviors, and digital control systems, are described in terms of such functional relations as observed in \cite{Bai}.
We also mention the papers \cite{Molica, Bisci, Cheng, Feng} for some interesting contributions related to some existence 
results for nonlinear algebraic systems, as well as the monographs \cite{Agarwal, Kelly} as general references for discrete problems.
         
The variational techniques employed in the discrete problems are the same  techniques already known for continuous problems with the necessary modifications. 
In order to establish existence and multiplicity of solutions for discrete problems, several authors exploited various  
methods such as fixed point theorems, critical point theory, and Brouwer degree, see for example \cite{Perera, Bereanu, Gao, Henderson}.

In 2008, Yang and Ji \cite{Yang} studied the structure of the spectrum of the problem 
      \[
      \begin{cases}
      u(i,i)+u(j,j)+\lambda a(i,j)u(i,j)=0,\quad\forall(i,j)\in [1,m]_{\mathbb{Z}}\times[1,n]_{\mathbb{Z}},\\
      u(i,0)=u(i,n+1)=0 ,\quad \forall i \in [1,m]_{\mathbb{Z}},\\
      u(0,j)=u(m+1,j)=0 ,  \quad \forall j \in [1,n]_{\mathbb{Z}},
      \end{cases}
      \]
and they found the existence of a positive eigenvector corresponding to the smallest eigenvalue. In 2010, Galewski and Orpel \cite{Galewski}, 
using variational methods and some monotonicity results, considered the  problem (\ref{eqn1.1}) without a weight, i.e., the problem 
       \begin{equation}\label{eqn1.3}
       \begin{cases}
       \Delta_{1}(\Delta_{1}u(i-1,j))+\Delta_{2}(\Delta_{2}u(i,j-1))+\lambda f((i,j),u(i,j))=0,\\
       \hspace*{2.43in}\forall (i,j)\in [1,m]_{\mathbb{Z}}\times[1,n]_{\mathbb{Z}},\\
       u(i,0)=u(i,n+1)=0 ,\quad \forall i \in [1,m]_{\mathbb{Z}},\\
       u(0,j)=u(m+1,j)=0 , \quad \forall j \in [1,n]_{\mathbb{Z}},
       \end{cases}
  \end{equation}
and they established the existence of one solution.     
Other works on the problem (\ref{eqn1.3}) can be found in \cite{Repovs, Imbesi} where the authors, using variational methods and maximum principle,
proved the existence of infinitely many solutions and  determined unbounded intervals of parameters such that (\ref{eqn1.3}) admits an unbounded sequence of solutions.
        
In this paper, motivated by this large interest, we study the existence of at least one  nontrivial solution of the problem (\ref{eqn1.1})
under some conditions on the nonlinearity function $f$ and for suitable values of the parameter $\lambda$. The tools employed include the  
theory of variational methods, the mountain pass theorem, and linking arguments.
     
The rest of this paper is organized as follows. In Section \ref{sec2}, we present some preliminaries that will be used in Section \ref{sec4}. 
In Section \ref{sec3}, we introduce some corresponding variational framework and define some functionals for the transformation of the problem (\ref{eqn1.1}).
In the last section, we give the main results and their proofs.

\setcounter{equation}{0}
\section{Preliminaries}
\label{sec2}

In this section, we present some definitions and theorems that will be used in the sequel. We refer the reader to \cite{Ambroseti, Ricceri, Struwe, Zhang} for more details.

\begin{definition}\label{def1}
Let E be a real Banach space, D an open subset of E. Suppose that a functional $\varphi:D\rightarrow\mathbb{R}$ is Fr\'{e}chet differentiable on D. 
If $u_{0}\in D$ and the Fr\'{e}chet derivative of $\varphi$ satisfies $\varphi\prime(u_{0})=0$, then we say that $u_{0}$ is a critical point of $\varphi$ 
and $\varphi(u_{0})$ is a critical value of $\varphi$.
\end{definition}

Let $C^{1}(E,\mathbb{R})$ denote the set of functionals that are Fr\'{e}chet differentiable in $E$ and their Fr\'{e}chet derivatives are continuous in $E$.

\begin{definition}\label{def2}
	Let E be a real Banach space and $\varphi\in C^{1}(E,\mathbb{R})$. We say that $\varphi$ satisfies the Palais-Smale condition ((PS)
	condition for short) if for every sequence  $(u_{n})\in E$ such that $\varphi(u_{n})$ is bounded and $\varphi\prime(u_{n})\rightarrow 0$ 
	as $n\rightarrow\infty$, there exists a subsequence of  $(u_{n})$ which is convergent in E.
\end{definition}

\begin{theorem}\label{theo2.3}(\!\!\cite{Struwe})
	Let E be a real Banach space and $\varphi :E\rightarrow\mathbb{R}$ 
	 is weakly lower semi-continuous function and coercive, i.e., $\displaystyle\lim_{\lVert x \rVert\rightarrow+\infty}\varphi(x)= +\infty$, 
	 then there exists $x_{0}\in E$ such that 
	\[
	\inf _{x\in E}\varphi(x)=\varphi(x_{0}).
	\]
    Furthermore, if $\varphi\in C^{1}(E,\mathbb{R})$, then $x_{0}$ is also a critical point of $\varphi$, i.e., $\varphi\prime(x_{0})=0$.		
\end{theorem}

\begin{theorem}(Mountain Pass Lemma \cite{Ambroseti})\label{theo2.4}
	Let E be a real Banach space and $\varphi\in C^{1}(E,\mathbb{R})$  satisfying the
	(PS) condition with $\varphi(0)=0$. Suppose that
\begin{itemize}
\item[(i)] There exists $\rho>0$ and $\alpha>0$ such that $\varphi(u)\geq\alpha$ for all $u\in E$, with $\lVert u\rVert=\rho$.

\item[(ii)] There  exists $u_{0}\in E$ with $\lVert u\rVert\geq\rho$ such that $\varphi(u_{0})<0$.

\end{itemize}
Then $\varphi$ has a critical value $c\geq\alpha$  and $c=\displaystyle\inf_{h\in\Gamma}\displaystyle\max_{s\in[0,1]}\varphi(h(s))$, 
where
\[
 \Gamma=\{h\in C([0,1],E) : h(0)=0,h(1)=u_{0}\}. \]
 
\end{theorem}

\begin{theorem}\label{theo2.5}(\!\!\cite{Ricceri})
	Let X be a reflexive real Banach space and let $\Phi,\Psi : X\rightarrow \mathbb{R} $ be two G\^{a}teaux differentiable functionals 
	such that $\Phi$ is strongly continuous, sequentially weakly lower semicontinuous and coercive in X and  $\Psi$ is sequentially weakly upper 
	semicontinuous in X. Let $J_{\lambda}$ be the functional defined as $J_{\lambda}:=\Phi - \lambda\Psi$, $\lambda\in \mathbb{R}$, and 
	for any $r>\displaystyle\inf_{X}\Phi$ let $\varphi$ be the function defined by 
	\[
	\varphi(r)=\displaystyle\inf_{u\in\Phi^{-1}((-\infty,r))}\dfrac{\displaystyle\sup_{v\in\Phi^{-1}((-\infty,r))}\Psi(v)-\Psi(u)}{r-\Phi(u)}.
	\] 
	 Then, for any $r>\displaystyle\inf_{X}\Phi $  and any $\lambda\in \left( 0, {1}/{\varphi(r)}\right)$, the restriction of 
	 the functional $J_{\lambda}$ to $\Phi^{-1}((-\infty,r))$ 
admits a global minimum, which is a critical point (precisely a local minimum) of $J_{\lambda}$ in X.

\end{theorem}

\setcounter{equation}{0}
\section{Variational framework}
\label{sec3}
	
In this section, we introduce the corresponding variational framework for the problem (\ref{eqn1.1}). 
Let $E$ be the $mn$ dimensional space $\mathbb{R}^{m}\times\mathbb{R}^{n}$ endowed by the norm
\[
\lVert u\rVert=\left( \displaystyle\sum_{i=1}^{m}\displaystyle\sum_{j=1}^{n}u^{2}(i,j)\right) ^{\frac{1}{2}}.
\]

For all $(i,j)\in[1,m]_{\mathbb{Z}}\times[1,n]_{\mathbb{Z}}$, the problem (\ref{eqn1.1}) can be rewritten as follows 
	\begin{equation}\label{eqn3.1}
	\begin{gathered}
	-p(i-1,j)u(i-1,j)+(p(i-1,j)+2p(i,j)+p(i,j-1))u(i,j)-p(i,j)u(i+1,j)\\
	\hspace*{1.56in}-p(i,j-1)u(i,j-1)- p(i,j)u(i,j+1)=\lambda f((i,j),u(i,j)),
	\end{gathered}
	\end{equation}
with the same boundary conditions as for the problem (\ref{eqn1.1}).
	
For $j\in [1,n]_{\mathbb{Z}}$, we let
\begin{equation*}
U_{j}=(u(1,j),u(2,j),\ldots,u(m,j))^{T} \quad  \text{and} \quad	U=(U_{1},U_{2}, \ldots ,U_{n})^{T},
\end{equation*} 
and for $U\in E$, we define
	\begin{eqnarray*}
		\textbf{H}(U) & = & (f((1,1),u(1,1)),f((2,1),u(2,1)),\ldots,f((m,1),u(m,1)),\\ 
		&& f((1,2),u(1,2)),\ldots,f((m,2),u(m,2)),\ldots,\\ 
		&& f((1,n),u(1,n)),\ldots,f((m,n)u(m,n)))^{T}.
	\end{eqnarray*}
Then, the problem (\ref{eqn1.1}) can be formulated as the nonlinear algebraic system
		\begin{equation}\label{eqn3.3}
		\textbf{\textit{M}}U=\lambda\textbf{\textit{H}}(U),
		\end{equation}
where $\textbf{\textit{M}}$ is an $mn\times mn$ matrix given by 
\begin{equation}\label{eqn3.4}
	\begin{pmatrix}
	L_{1}&-P_{1} & 0& 0&\dots& 0& 0& 0& 0\\
	-P_{1}& L_{2}&-P_{2}& 0&\dots& 0& 0& 0& 0\\
	0&-P_{2}& L_{3} &-P_{3}&\dots& 0& 0& 0& 0\\
	0& 0& -P_{3}& L_{4}&\dots& 0& 0& 0& 0\\
	\dots&\dots&\dots&\dots&\dots&\dots&\dots&\dots&\dots\\
	\dots&\dots&\dots&\dots&\dots&\dots&\dots&\dots&\dots\\
	0& 0& 0& 0&\dots& L_{n-3}&-P_{n-3} & 0& 0\\
	0& 0& 0& 0&\dots&-P_{n-3}& L_{n-2}&-P_{n-2}& 0\\
	0& 0& 0& 0&\dots& 0&-P_{n-2}& L_{n-1}&-P_{n-1}\\
	0& 0& 0& 0&\dots& 0& 0&-P_{n-1}& L_{n}	
	\end{pmatrix},
	\end{equation} 
with, for all $j\in [1,n]_{\mathbb{Z}}$, $L_{j}=(l^{j}_{kl})_{m\times m}$ being an $m\times m$ symmetric tridiagonal matrix defined by
\begin{equation}\label{eqn3.5}
l^{j}_{kl} = \left\{
 \begin{array}{ll}
p(k-1,j)+2p(k,j)+p(k,j-1) & \textrm{if $k=l$},\\
l^{j}_{k,k-1}=-p(k,j)=l^{j}_{k,k+1},\\
0 & \textrm{elsewhere},
\end{array}
 \right.
\end{equation}
and, for all $j\in [1,n-1]_{\mathbb{Z}}$, $ P_{j} $ being an  $ m\times m $ diagonal matrix given by 
\begin{equation}\label{eqn3.6}
	P_{j}=	
	\begin{pmatrix}
		p(1,j)& 0 & \dots&\dots&\dots& \dots&\dots& \dots\\
		0 & p(2,j)& 0 & \dots&\dots&\dots& \dots&\dots& \\
		\dots& 0 & p(3,j)& 0 &\dots&\dots&\dots&\dots\\
		\dots&\dots& 0&\dots&\dots&\dots&\dots&\dots\\
		\dots&\dots&\dots&\dots&\dots& \dots & 0&\dots\\
		\dots& \dots& \dots& \dots& \dots& 0 & p(m-1,j)& 0\\
		\dots& \dots& \dots& \dots& \dots& \dots& 0 & p(m,j)	
	\end{pmatrix}.
\end{equation}	
	
For all $\lambda>0$, we let $I_{\lambda}:E \rightarrow \mathbb{R}$ be the functional defined by
\begin{equation}\label{eqn3.7}
I_{\lambda}(U)=\dfrac{1}{2}U^{T}\textbf{\textit{M}}U- \lambda\displaystyle\sum_{i=1}^{m}\displaystyle\sum_{j=1}^{n}F((i,j),u(i,j)),
\end{equation}
where 
\begin{equation}\label{eqn3.8 }
F((i,j),x)=\int_{0}^{x}f((i,j),t)dt.
\end{equation}		

For $U\in E$, we define two reals functionals $\phi$ and $\psi$ by 
\begin{equation}\label{eqn3.9}
\phi(U)=\dfrac{1}{2}U^{T}\textbf{\textit{M}}U,
\end{equation}
and  
\begin{equation}\label{eqn3.10}
\psi(U)=\displaystyle\sum_{i=1}^{m}\displaystyle\sum_{j=1}^{n}F((i,j),u(i,j)).
\end{equation}	
Then, the functional $I_{\lambda}$ can be rewritten as follows 
\begin{equation}\label{eqn3.11}
I_{\lambda}(U)=\phi(U)-\lambda\psi(U),  \qquad \forall U\in E.
\end{equation}	
Standard argument assures that, with any fixed $\lambda>0$,  the functional $ I_{\lambda}$ is G\^{a}teaux differentiable with  
\begin{equation}\label{eqn3.12}
I^{'}_{\lambda}(U)=\textbf{\textit{M}}U- \lambda\textbf{\textit{H}}(U) ,\quad \forall U\in E.
\end{equation}
It is clear that $U$ is a solution of (\ref{eqn1.1}), if and only if $U$  is a critical point of the functional $I_{\lambda}$.
Thus, the search of solutions of the problem (\ref{eqn1.1}) reduces to finding the critical points $U\in E$  of the functional $I_{\lambda}$.

Now, we prove the following lemma.

\begin{lemma}\label{lem3.1}
M  is a positive definite matrix.     
\end{lemma}

\begin{proof}
For $j\in [1,n]_{\mathbb{Z}} $, we let $ X_{j}^{T}=(x_{1,j}, x_{2,j}, x_{3,j},...,x_{m,j}) \in\mathbb{R}^{m}$. 
For each $j\in [1,n]_{\mathbb{Z}} $, $L_{j}$ is a real symmetric matrix, then 
\begin{eqnarray*}
 X_{j}^{T}L_{j}X_{j} & = & \displaystyle\sum_{i=1}^{m}(p(i-1,j)+2p(i,j) +p(i,j-1) )x_{i,j}^{2} - 2\displaystyle\sum_{i=1}^{m-1}p(i,j)x_{i,j}x_{i+1,j}\\
  & & =\displaystyle\sum_{i=1}^{m}p(i-1,j)x_{i,j}^{2} +2\displaystyle\sum_{i=1}^{m}p(i,j)x_{i,j}^{2} + \displaystyle\sum_{i=1}^{m}p(i,j-1)x_{i,j}^{2} \\
  & & -2\displaystyle\sum_{i=1}^{m-1}p(i,j)x_{i,j}x_{i+1,j}\\
  & &=\displaystyle\sum_{i=0}^{m-1}p(i,j)x_{i+1,j}-2\displaystyle\sum_{i=1}^{m-1}p(i,j)x_{i,j}x_{i+1,j} + \displaystyle\sum_{i=1}^{m}p(i,j)x_{i,j}^{2}\\
  & & + \displaystyle\sum_{i=1}^{m}p(i,j)x_{i,j}^{2} + \displaystyle\sum_{i=1}^{m}p(i,j-1)x_{i,j}^{2}\\
  & &  =\displaystyle\sum_{i=1}^{m-1}p(i,j)(x_{i+1,j}-x_{i,j})^{2} + p(0,j)x_{1,j}^{2}+ p(m,j)x_{m,j}^{2} + \displaystyle\sum_{i=1}^{m}p(i,j)x_{i,j}^{2}\\
  & & + \displaystyle\sum_{i=1}^{m}p(i,j-1)x_{i,j}^{2}.
\end{eqnarray*}
Thus,
\begin{equation}\label{eqn3.13}
X_{j}^{T}L_{j}X_{j}\geq \displaystyle\sum_{i=1}^{m}(p(i,j) + p(i,j-1))x_{i,j}^{2}. 
\end{equation}
On the other hand, for any $X=(X_{1},X_{2}, ...,X_{n})\in\mathbb{R}^{mn} $, we have
\[
X^{T}\textbf{\textit{M}}X=\displaystyle\sum_{j=1}^{n}X_{j}^{T}L_{j}X_{j} -2\displaystyle\sum_{j=1}^{n-1}\displaystyle\sum_{i=1}^{m}p(i,j)x_{i,j}x_{i,j+1}.
\]
In view of (\ref{eqn3.13}), we deduce that
\begin{eqnarray*}
X^{T}\textbf{\textit{M}}X & \geq & \displaystyle\sum_{j=1}^{n}\displaystyle\sum_{i=1}^{m}(p(i,j) + p(i,j-1))x_{i,j}^{2} - 2\displaystyle\sum_{j=1}^{n-1}\displaystyle\sum_{i=1}^{m}p(i,j)x_{i,j}x_{i,j+1}\\
& & \geq\displaystyle\sum_{j=1}^{n}\displaystyle\sum_{i=1}^{m}p(i,j)x_{i,j}^{2}+ \displaystyle\sum_{j=1}^{n}\displaystyle\sum_{i=1}^{m}p(i,j-1)x_{i,j}^{2} \\
& & - 2\displaystyle\sum_{j=1}^{n-1}\displaystyle\sum_{i=1}^{m}p(i,j)x_{i,j}x_{i,j+1}\\
& & \geq\displaystyle\sum_{j=1}^{n-1}\displaystyle\sum_{i=1}^{m}p(i,j)x_{i,j}^{2}+\displaystyle\sum_{i=1}^{m}p(i,n)x_{i,n}^{2}+ \displaystyle\sum_{j=0}^{n-1}\displaystyle\sum_{i=1}^{m}p(i,j)x_{i,j+1}^{2} \\
& & - 2\displaystyle\sum_{j=1}^{n-1}\displaystyle\sum_{i=1}^{m}p(i,j)x_{i,j}x_{i,j+1}\\
& &\geq\displaystyle\sum_{j=1}^{n-1}\displaystyle\sum_{i=1}^{m}p(i,j)(x_{i,j}^{2} + x_{i,j+1}^{2}-2x_{i,j}x_{i,j+1})+ \displaystyle\sum_{i=1}^{m}p(i,n)x_{i,n}^{2}\\
& & + \displaystyle\sum_{i=1}^{m}p(i,0)x_{i,1}^{2}.
\end{eqnarray*}
Then, taking into account that $ p(i,0)=0$ for all $i \in [1,m]_{\mathbb{Z}}$, we obtain that 
\begin{equation}\label{eqn3.14}
X^{T}\textbf{\textit{M}}X \geq\displaystyle\sum_{j=1}^{n-1}\displaystyle\sum_{i=1}^{m}p(i,j)(x_{i,j} - x_{i,j+1})^{2} +\sum_{i=1}^{m}p(i,n)x_{i,n}^{2} .
\end{equation}
Therefore, for any $X\in\mathbb{R}^{mn}$, we get that $X^{T}\textbf{\textit{M}}X \geq 0$, and if $X^{T}\textbf{\textit{M}}X =0$, 
the inequality (\ref{eqn3.14}) indicates that
$X_{j}=X_{j+1}$ for all $j \in [1,n-1]_{\mathbb{Z}}$ and $X_{n}=0$, so $X=0_{E}$.
Hence, we deduce  that $X^{T}\textbf{\textit{M}}X>0$ for all $X\in\mathbb{R}^{mn}$ with $X\neq0_{E}$, so $\textbf{\textit{M}}$ is a positive definite matrix.
\end{proof}
 
We let,  $\lambda_{1}$, $\lambda_{2}$, $\lambda_{3}$, $\ldots$,  and  $\lambda_{mn}$ be the eigenvalues of the positive definite matrix \textbf{\textit{M}} 
ordered as follows
\[
0<\lambda_{1}\leq\lambda_{2}\leq\ldots\leq\lambda_{mn}.
\]
It is  easy to show that, for every $U\in E$, we have 
\begin{equation}\label{eqn3.15}
\dfrac{1}{2}\lambda_{1}\lVert U\rVert^{2}\leq \phi(U)\leq\dfrac{1}{2}\lambda_{mn}\lVert U\rVert^{2},
\end{equation}
and 
\begin{equation}\label{eqn3.16}
\lVert U\rVert_{\infty}^{2}\leq \dfrac{2}{\lambda_{1}}\phi(U),
\end{equation}
where $\lVert U\rVert_{\infty}=\max\{\lvert u(i,j)\rvert$ , $(i,j)\in[1,m]_{\mathbb{Z}}\times[1,n]_{\mathbb{Z}}\}$.

\setcounter{equation}{0}
\section{Existence  results and their proofs}
\label{sec4}

In this section, we use the variational techniques mentioned in Section \ref{sec2} to show the existence of solutions of the 
problem (\ref{eqn1.1}).  

\begin{theorem}\label{theo4.1}
Assume that the following condition holds:
\begin{itemize}	
 \item[(H{1})] $\displaystyle\lim_{t\rightarrow 0}\dfrac{F((i,j),t)}{t^{2}}=+\infty$, $\forall (i,j)\in [1,m]_{\mathbb{Z}}\times[1,n]_{\mathbb{Z}}$.
\end{itemize}
Then there exists $\lambda^{\star}>0$ such that, for each  $\lambda\in (0,\lambda^{\star})$, the  problem (\ref{eqn1.1}) has at least one nontrivial solution. 
\end{theorem}

\begin{proof}
We will use the version of Ricceri's variational principle given in Theorem \ref{theo2.5}. 
Firstly, the functionals $\phi$  and $\psi$ defined in  (\ref{eqn3.9}) and (\ref{eqn3.10}) are G\^{a}teaux differentiable, 
and since $E$ is a finite dimensional space, they  satisfy all regularity assumptions of Theorem \ref{theo2.5}. 
The inequality (\ref{eqn3.15}) yields that  $\phi$ is coercive.

Secondly, let $\alpha>0$ and put $r=\dfrac{\lambda_{1}}{2}\alpha^{2}$, then for all $U\in E$ such that $\phi(U)<r$, taking (\ref{eqn3.16}) into account,  
we get that  $\lVert U\rVert_{\infty}<\alpha$.

For all $U\in E$ such that $\phi(U)<r$, by  (\ref{eqn3.10}), we have 
   \[
   \psi(U)\leq \displaystyle\sum_{i=1}^{m}\displaystyle\sum_{j=1}^{n}\displaystyle\max_{\lvert t\rvert\leq\alpha}F((i,j),t),
   \]  
which yields that
\begin{equation}\label{eqn4.1}
\displaystyle\sup_{\phi(U)<r}\psi(U)\leq \displaystyle\sum_{i=1}^{m}\displaystyle\sum_{j=1}^{n}\displaystyle\max_{\lvert t\rvert\leq\alpha}F((i,j),t).
\end{equation}
On the other hand, we let 
\begin{equation}\label{eqn4.2}
\lambda^{\star}= \dfrac{\lambda_{1}\alpha^{2}}{2\displaystyle\sum_{i=1}^{m}\displaystyle\sum_{j=1}^{n}\displaystyle\max_{\lvert t\rvert\leq\alpha}F((i,j),t)}>0
\end{equation}
and 
\begin{equation}\label{eqn4.3}
	\varphi(r):=\displaystyle\inf_{u\in\phi^{-1}((-\infty,r))}\dfrac{\displaystyle\sup_{v\in\phi^{-1}((-\infty,r))}\psi(v)-\psi(u)}{r-\phi(u)}.
\end{equation}
One has
\[ \varphi(r)\leq\dfrac{\displaystyle\sup_{v\in\phi^{-1}((-\infty,r))}\psi(v)-\psi(u)}{r-\phi(u)}\leq\dfrac{\displaystyle\sup_{v\in\phi^{-1}((-\infty,r))}\psi(v)}{r},
\]
then using (\ref{eqn4.1}), we have 
\[
\varphi(r)\leq\dfrac{1}{r}\displaystyle\sum_{i=1}^{m}\displaystyle\sum_{j=1}^{n}\displaystyle\max_{\lvert t\rvert\leq\alpha}F((i,j),t),
\]
therefore, 
\[
\lambda^{\star}\leq\dfrac{1}{\varphi(r)}.
\]
By Theorem \ref{theo2.5}, we see that, for every $\lambda\in (0,\lambda^{\star})$, the functional $I_{\lambda}$  admits at least one critical point  $U_{\lambda}\in\phi^{-1}((-\infty,r))$.

Next, it remains to show that $U_{\lambda}\neq 0_{E}$, if $f((i,j),0)\neq 0$ for some $(i,j)\in [1,m]_{\mathbb{Z}}\times[1,n]_{\mathbb{Z}}$. 
Since the trivial vector $0_{E}$ does  not solve problem (\ref{eqn1.1}),  $U_{\lambda}\neq 0_{E}$.
 
For the other case when $f((i,j),0)=0$ for every $(i,j)\in [1,m]_{\mathbb{Z}}\times[1,n]_{\mathbb{Z}}$, by the condition (H{1}), 
we can fix a sequence $ \{ u_{p}\}\subset\mathbb{R}^{+}$ converging to zero. Then, one has 
\[
\displaystyle\lim_{p\rightarrow +\infty}\dfrac{F((i,j),u_{p})}{u_{p}^{2}}=+\infty,   \quad    \forall (i,j)\in [1,m]_{\mathbb{Z}}\times[1,n]_{\mathbb{Z}},
\] 
and for a fixed constant $a>0$, there exists $\rho>0$ such that, $F((i,j),t)>at^{2}$  for all $(i,j)\in [1,m]_{\mathbb{Z}}\times[1,n]_{\mathbb{Z}}$ and $\lvert t\rvert\leq\rho$.
Let $V\in E$ with $v(i,j)=1$ for all $(i,j)\in [1,m]_{\mathbb{Z}}\times[1,n]_{\mathbb{Z}}$, and set $w_{p}=u_{p}V$ for any $p \in \mathbb{N}$. It is clear that $w_{p}\in E$ and 
$\lVert w_{p}\rVert =\lvert u_{p}\rvert  \lVert V\rVert\rightarrow 0$ as $p\rightarrow+\infty$.
Then, for $p$ large enough, we have $ \lVert w_{p}\rVert<\sqrt{\dfrac{\lambda_{1}}{\lambda_{mn}}}\alpha$,  furthermore $\phi(w_{p})<r$, so $w_{p}\in\phi^{-1}((-\infty,r)$.
Therefore, 
\[
\dfrac{\psi(w_{p})}{\phi(w_{p})}=\dfrac{\displaystyle\sum_{i=1}^{m}\displaystyle\sum_{j=1}^{n}F((i,j),u_{p}v(i,j))}{u_{p}^{2}\phi(V)}\geq 
\dfrac{a u_{p}^{2}\displaystyle\sum_{i=1}^{m}\displaystyle\sum_{j=1}^{n}v(i,j)^{2}}{u_{p}^{2}\phi(V)}=\dfrac{amn}{\phi(V)},
\]
for $p$ sufficiently large.

Let $A>0$ arbitrary large enough, and choose $a$ such that $A<\dfrac{amn}{\phi(V)}$, then for  $p$  large enough, one has 
\[
\dfrac{\psi(w_{p})}{\phi(w_{p})}>A.
\]
Then, $\displaystyle\limsup_{p\rightarrow+\infty}\dfrac{\psi(w_{p})}{\phi(w_{p})}=+\infty$.
Hence, for $p$ sufficiently large and $\lambda>0$, we deduce that $I_{\lambda}(w_{p})= \phi(w_{p})-\lambda\psi(w_{p})<0$.
Since $U_{\lambda}$ is a global minimum of the function $I_{\lambda}$ in $\phi^{-1}((-\infty,r))$ and $w_{p}\in\phi^{-1}((-\infty,r))$, we get that
\[
I_{\lambda}(U_{\lambda}) \leq I_{\lambda}(w_{p})<0=I_{\lambda}(0_{E}),
\]
so $U_{\lambda}\neq 0_{E}$.  The proof is complete.
\end{proof}

\begin{theorem}\label{theo4.2}
Assume that the following assumptions holds:
\begin{itemize}

\item[(H2)] there exist two real constants $c>0$  and $\eta>0$, such that
 $$
 F((i,j),t)<-ct^{2},\quad \forall (i,j)\in [1,m]_{\mathbb{Z}}\times[1,n]_{\mathbb{Z}} \quad and \quad \lvert t\rvert <\eta;
 $$

\item[(H3)] there exist  real constants $ a, b, T, \alpha$ such that $a>0$, $T>0$,  and  $1<\alpha<2$ such that
$$F((i,j),t)<a\lvert t\rvert^{\alpha}+b,\quad\forall (i,j)\in [1,m]_{\mathbb{Z}}\times[1,n]_{\mathbb{Z}}\quad and \quad \lvert t\rvert \geq T.$$
\end{itemize}
Then, for any  parameter $\lambda \in \left( \dfrac{\lambda_{mn}}{2c}, +\infty\right) $, the problem (\ref{eqn1.1}) has at least one nontrivial solution. 
\end{theorem}

\begin{proof}

Let  $U\in E$ such that $\lVert U\rVert $ is large enough. From (\ref{eqn3.10}) and according to the conditions (H{3}), we have 
\[
\psi(U)\leq a\displaystyle\sum_{i=1}^{m}\displaystyle\sum_{j=1}^{n}\lvert u(i,j)\rvert^{\alpha} + mnb.
\]
By the H\"{o}lder inequality, we get that
\begin{eqnarray*}
\psi(U) & \leq & an^{\frac{2-\alpha}{2}}\displaystyle\sum_{i=1}^{m}(\displaystyle\sum_{j=1}^{n}\lvert u(i,j)\rvert^{2})^{\frac{\alpha}{2}} + mnb.\\
       & &  \leq a(mn)^{\frac{2-\alpha}{2}}(\displaystyle\sum_{i=1}^{m}\displaystyle\sum_{j=1}^{n}\lvert u(i,j)\rvert^{2})^{\frac{\alpha}{2}} + mnb\\
       & & \leq a(mn)^{\frac{2-\alpha}{2}}\lVert U\rVert^{\alpha} + mnb.        
\end{eqnarray*}
Then, owing to (\ref{eqn3.11}) and from (\ref{eqn3.15}), one immediately has 
\[
I_{\lambda}(U)\geq\dfrac{\lambda_{1}}{2}\lVert U\rVert^{2} - a(mn)^{\frac{2-\alpha}{2}}\lambda\lVert U\rVert^{\alpha}-mnb\lambda, 
\]
for any $U\in E$ with $\lVert U\rVert$ is large enough.

Since $1<\alpha<2$, $I_{\lambda}(U)\rightarrow +\infty$ as $\lVert U\rVert\rightarrow +\infty$, which implies that the functional  
$I_{\lambda}$ is coercive. Since $f((i,j),.)$ is continuous for all $(i,j)\in [1,m]_{\mathbb{Z}}\times[1,n]_{\mathbb{Z}}$, 
then $ I_{\lambda}$ is continuous and bounded from below. Therefore, by Theorem \ref{theo2.3}, we deduce that $I_{\lambda}$  
attains its minimum at some point $\tilde{U_{\lambda}}\in E$ which is also the critical point of $I_{\lambda}$.

 On the other hand, we will show that $\tilde{U_{\lambda}}\neq0_{E}$. Let $\lambda \in \left( \dfrac{\lambda_{mn}}{2c}, +\infty\right)$ and $U\in E$ 
 such that $\lvert u(i,j)\rvert<\eta$,  $\forall(i,j)\in [1,m]_{\mathbb{Z}}\times[1,n]_{\mathbb{Z}}$. According to (H{2}), we have   
\[
  F((i,j),u(i,j)) \leq- c \lvert u(i,j)\rvert^{2},   \quad   \forall(i,j)\in [1,m]_{\mathbb{Z}}\times[1,n]_{\mathbb{Z}}.
\]
Then, for one $U\in E$ such that $\lVert U\rVert =\eta'$, where $\eta'=\eta\sqrt{mn}$, the relations (\ref{eqn3.10}) and (\ref{eqn3.11}) give  
 \[
  \psi(U)\leq-c\lVert U\rVert^{2}
\]
and 

\[
I_{\lambda}(U)\leq \left( \dfrac{\lambda_{mn}}{2} - \lambda c\right) \lVert U\rVert^{2}.
\]
Then by the definition of $\tilde{U_{\lambda}}$, we prove that $I_{\lambda}(\tilde{U_{\lambda}})\leq \left( \dfrac{\lambda_{mn}}{2} - \lambda c\right) \eta'<0$,
which implies that  $\tilde{U_{\lambda}}\neq0_{E}$. The proof is complete.
\end{proof}
 
\begin{theorem}\label{theo4.3}
Suppose that the condition (H{2}) is satisfied and suppose additionally that
\begin{itemize}
\item[(H{4})] there exist $A>0$ such that 
$$
\displaystyle\lim_{\lvert t\rvert\rightarrow \infty}\sup\dfrac{F((i,j),t)}{t^{2}}<A,\quad \forall (i,j)\in [1,m]_{\mathbb{Z}}\times[1,n]_{\mathbb{Z}}.
$$
\end{itemize}
Then, for each $\lambda\in \left( 0, \dfrac{\lambda_{1}}{2A}\right) $ the  problem (\ref{eqn1.1}) has at least one nontrivial  solution.  
\end{theorem}

\begin{proof}
Firstly, we show that the functional $ I_{\lambda}$ is coercive. The assumption (H{4}) yields the existence of a constant $C>0$ such that 
\[
F((i,j),t)<At^{2}, \quad \forall  \lvert t\rvert>C  \quad \text{and}   \quad  \forall(i,j)\in [1,m]_{\mathbb{Z}}\times[1,n]_{\mathbb{Z}}.
\] 
For $U\in E$ sufficiently large (taking $\lvert u(i,j)\rvert>C$), from  (\ref{eqn3.11}) and (\ref{eqn3.15}), it follows that
\[
I_{\lambda}(U)\geq\left( \dfrac{\lambda_{1}}{2}-\lambda A\right) \lVert U\rVert^{2}.
\]
Then, for all $\lambda<\dfrac{\lambda_{1}}{2A}$, we obtain $I_{\lambda}(U)\rightarrow +\infty$ as $\lVert U\rVert\rightarrow +\infty$, so $ I_{\lambda}$ is coercive. 
 Since $f((i,j),.)$ is continuous, then $ I_{\lambda}$ is weakly continuous and G\^{a}teaux differentiable, therefore according to Theorem \ref{theo2.3},  
 we deduce that the functional $ I_{\lambda}$ admits a critical point $\tilde{U}$.

Arguing as in the proof of Theorem \ref{theo4.1}, we get that $\tilde{U}\neq 0_{E}$. The proof is complete.
\end{proof}

\begin{theorem}\label{theo4.4}
Assume that the following assumptions holds
\begin{itemize}

\item[(H{5})] there exist two functions  $\alpha  :[1,m]_{\mathbb{Z}}\times[1,n]_{\mathbb{Z}}\rightarrow (0 , +\infty) $, $\beta   :[1,m]_{\mathbb{Z}}\times[1,n]_{\mathbb{Z}}\rightarrow \mathbb{R} $ 
and a constant $M>0$ such that
	$$  
	F((i,j),t)\geq \alpha(i,j)t^{2} +\beta(i,j) , \quad \forall (i,j)\in [1,m]_{\mathbb{Z}}\times[1,n]_{\mathbb{Z}} ,\quad \lvert t\rvert>M;
	$$

\item[(H{6})] $\displaystyle\lim_{\lvert t\rvert\rightarrow 0}\dfrac{F((i,j),t)}{t^{2}}=0,\quad \forall (i,j)\in [1,m]_{\mathbb{Z}}\times[1,n]_{\mathbb{Z}}.$

\end{itemize}	
Then, for each $\lambda >\dfrac{\lambda_{mn}}{2\alpha^{-}}$, the  problem (\ref{eqn1.1}) has at least one nontrivial solution, 
where 
$$
\alpha^{-}=\min\{ \alpha(i,j);(i,j)\in [1,m]_{\mathbb{Z}}\times[1,n]_{\mathbb{Z}} \}.
$$
\end{theorem}

\begin{proof}
Fix $\lambda>\dfrac{\lambda_{mn}}{2\alpha^{-}}$. Firstly, we will check that $I_{\lambda}$ satisfies the PS condition. 
Let $\{ u_{n}\}\subset E$ be a sequence such $I_{\lambda}(u_{n})$ is bounded and $I'_{\lambda}(u_{n})\rightarrow 0$ as $n\rightarrow +\infty$, then there exists a constant $B>0$ 
such that $\lVert I_{\lambda}(u_{n})\rVert\leq B$. By (\ref{eqn3.10}), and from condition (H{5}), we infer that
   \begin{equation}\label{eqn4.4} 
   \psi (u_{n})\geq \alpha^{-}\lVert u_{n}\rVert^{2}+mn\beta^{-}.
   \end{equation}
Therefore, by (\ref{eqn3.11})-(\ref{eqn3.15}) and from (\ref{eqn4.4}), it follows that
\begin{equation}\label{eqn4.5}
-B\leq I_{\lambda}(u_{n})\leq\left( \dfrac{\lambda_{mn}}{2}-\lambda\alpha^{-}\right)  \lVert u_{n}\rVert^{2} - mn\lambda\beta^{-}, \quad \forall n \in \mathbb{N},
\end{equation}
so, for any $n \in \mathbb{N}$,
\[
\left( \lambda\alpha^{-} - \dfrac{\lambda_{mn}}{2}\right)  \lVert u_{n}\rVert^{2}\leq B - mn\lambda\beta^{-}.
\]
Since $\lambda>\dfrac{\lambda_{mn}}{2\alpha^{-}}$, $\{ u_{n}\}$ is a bounded sequence in $E$, which is a mn-dimensional space.
Thus, $\{ u_{n}\}$ possesses a convergent subsequence, this prove that $I_{\lambda}$ satisfies the PS condition.

Next, we need to prove the assumption (i) of Theorem  \ref{theo2.4}. In fact, from  (H{6}), there exists a constant $\mu>0$ such that 
\[
\lvert F((i,j),t) \rvert\leq\dfrac{\lambda_{1}}{4}t^{2} ,\quad \forall \lvert t\rvert\leq\mu \quad and \quad \forall (i,j)\in [1,m]_{\mathbb{Z}}\times[1,n]_{\mathbb{Z}}.
\]
Then, for any $U\in E$, with $\lVert U\rVert\leq\mu$ and from (\ref{eqn3.10})-(\ref{eqn3.15}), we have
\begin{equation}\label{eqn4.6}
I_{\lambda}(U)\geq \dfrac{\lambda_{1}}{2}\lVert U\rVert^{2} - \dfrac{\lambda_{1}}{4}\lVert U\rVert^{2}=\dfrac{\lambda_{1}}{4}\lVert U\rVert^{2}.
\end{equation}

Let $B_{\mu}=\{U \in E : \lVert U\rVert\leq\mu\}$ and take $\delta=\dfrac{\lambda_{1}}{4}\mu^{2}$, then one has 
\[
I_{\lambda}(U)\geq\delta>0 , \quad \forall U\in \partial B_{\mu}.
\]
Thus, the assumption (i) of Theorem \ref{theo2.4} is satisfied. It remains to show the  assumption (ii)  of Theorem \ref{theo2.4}.
For this, let $U^{\ast}$ be such that $\lVert U^{\ast} \rVert =1$ and a large enough real $t$. 
By (\ref{eqn4.5}), one has 
\[
I_{\lambda}(tU^{\ast})\leq(\dfrac{\lambda_{mn}}{2}-\lambda\alpha^{-})\lVert tU^{\ast}\rVert^{2} - mn\lambda\beta^{-}=(\dfrac{\lambda_{mn}}{2}-\lambda\alpha^{-})t^{2} - mn\lambda\beta^{-}.
\]
Since $\lambda>\dfrac{\lambda_{mn}}{2\alpha^{-}}$, we have $I_{\lambda}(tU^{\ast})\rightarrow -\infty$ as $t\rightarrow +\infty$, so for $t_{0}>\mu$, we have 
$\hat{U}=t_{0}U^{\ast}\in E\setminus B_{\mu}$ and $I_{\lambda}(\hat{U})<0$, which yield our conclusion.
 
Finally, our aim is to apply the  Theorem \ref{theo2.4}. Then, there exists at least one critical value $C\geq\delta\delta>0$ to $I_{\lambda}$. 
If we note that  $U_{\lambda}$ is the critical point associated with the value $C$, we have $I_{\lambda}(U_{\lambda})=C$, so $U_{\lambda}$ is a solution to the problem (\ref{eqn1.1}). 
Since  $I_{\lambda}(0_{E})=0$ and $C>0$ then $U_{\lambda}$. The proof is complete.
\end{proof}

\end{document}